\documentclass[article]{elsarticle}
\usepackage{amsfonts,latexsym}
\usepackage{epsfig}
\usepackage{amsmath}
\usepackage{graphicx}
\usepackage{ulem}
\usepackage{epsfig}

\usepackage{amssymb} 
\usepackage[utf8]{inputenc}

\usepackage[usenames,dvipsnames]{pstricks}
\usepackage{pst-grad} 
\usepackage{pst-plot} 
\usepackage{mathptmx}      
\usepackage{amsmath}
\usepackage{latexsym}
\usepackage{subfig}
\usepackage{epstopdf}

\usepackage{fancyhdr}
\usepackage{algorithm,algorithmic}
\usepackage{color}
\usepackage{mathrsfs,tabularx}

\newtheorem{proposition}{Proposition}

\newdefinition{definition}{Definition}
\newtheorem{theorem}{Theorem}

\newproof{proof}{Proof}
\newdefinition{remark}{Remark}

\begin{document}
	\title{Low rank approximate solutions to large-scale differential matrix Riccati equations}
	
\author[istanbul]{Y. G\"{u}ldo\u{g}an}
\ead{guldogan@yildiz.edu.tr}
\author[lille]{M. Hached \corref{cor1}}
\ead{mustapha.hached@univ-lille1.fr}
\author[calais]{K. Jbilou }
\ead{jbilou@lmpa.univ-littoral.fr}
\author[istanbul]{M. Kurulay}
\ead{mkurulay@yildiz.edu.tr}

	\cortext[cor1]{Corresponding author}
	
	\begin{abstract}
		In the present paper, we consider large-scale continuous-time differential matrix Riccati equations. To the authors' knowledge, the two main approaches proposed in the litterature are based on a splitting scheme or on a Rosenbrock / Backward Differentiation Formula (BDF)  methods.  The approach we propose is based on the reduction of the problem dimension prior to integration. We project the initial problem onto an extended block Krylov subspace and obtain a low-dimensional differential matrix Riccati equation. The latter matrix differential problem is then solved by a Backward Differentiation Formula (BDF) method and the obtained solution is used to reconstruct an approximate solution of the original problem. This process is repeated, increasing the dimension of the projection subspace until achieving a chosen accuracy. We give some theoretical results and a simple expression of the residual  allowing the implementation of a stop test in order to limit the dimension of the projection space. Some numerical experiments will be given.
	\end{abstract}

		\cortext[cor1]{Corresponding author}

		
			\address[istanbul]{Department of Mathematical Engineering, Yildiz Technical University, Davutpasa Kamp\"{u}s\"{u}, Istanbul, Turkey}	
		\address[lille]{Laboratoire Painlev\'e UMR 8524 (ANO-EDP), UFR Math\'ematiques, Universit\'e des Sciences et Technologies de Lille, IUT A D\'epartement Chimie, Rue de la Recherche (lieu-dit "Le Recueil"),
			BP 179 - 59653 Villeneuve d'Ascq Cedex, France}
		\address[calais]{Laboratoire de Math\'ematiques Pures et
		Appliqu\'ees, Universit\'e du Littoral C\^ote d'Opale, 50 Rue F. Buisson, BP 699 - 62228
		Calais cedex, France.}	
	
\maketitle

\date{}

\vspace{-6pt}

\section{Introduction}
In this paper, we consider the  continuous-time differential  matrix Riccati equation (DRE in short) on the time interval $[0,T_f]$ of the form

\begin{equation}\label{ric1}
\left\{\begin{aligned}
\dot X(t) & =A^T\,X(t)+X(t)\,A-X(t)\,B\,B^T\,X(t)+C^T\,C \\
X(0)&=X_0  \\
\end{aligned}
\right.
\end{equation}
where  $X_0$ is some given $n \times n$ low-rank matrix, $X$ is the unknown matrix function, 
$ A \in \mathbb{R}^{n \times n} $ is assumed to be large, sparse and nonsingular, $ B \in
\mathbb{R}^{n \times \ell} $ and $ C \in \mathbb{R}^{s \times n} $. The matrices $ B $ and $ C $ are assumed to have full rank with  $\ell, s \ll n $. \\
Differential Riccati equations play a fundamental role in many areas such as control,
filter design theory, model reduction problems, differential
equations and robust control problems \cite{abou03,lancaster95,mehrmann91}. Differential  matrix Riccati  equations are also involved in game theory, wave propagation and scattering theory such as boundray value problems; see \cite{abou03,camo}. In the last decades,  some numerical methods have been proposed for approximating solutions of large scale algebraic Riccati equations \cite{benner98,heyouni09,jbilou03,jbilou06,simoncini07}.\\
Generally, the matrices $A$, $B$ and $C$ are obtained from the
discretization of operators defined on infinite dimensional
subspaces. Moreover, the matrix $A$ is generally sparse, banded
and very large. For such problems, only a few attempts have been made to solve Equation (\ref{ric1}), see \cite{benner04} for instance. \\

\noindent In the present paper, we propose a projection method onto Krylov subspaces. The idea is to project the initial differential Riccati equation onto an extended block Krylov subspace of small dimension, solve the  obtained  low dimensional differential matrix equation and get  approximate solutions to the initial differential matrix equations. \\

\noindent An  expression of the solution of equation \eqref{ric1} is avalaible under some assumptions on the coefficient matrices $A$, $B$ and $C$, see \cite{anderson71} for more details. This result can be stated as follows.

\begin{theorem}\label{exactformula}
Assuming that $(A,B)$ is stabilizable and $(C,A)$ is observable and provided that $X(0)>0$, the differential Riccati equation \eqref{ric1} admits a unique solution $X$ given by
\begin{equation}\label{solexacte1}
X(t)=\widetilde{X}+e^{t \widetilde{A}^T }[e^{t \widetilde{A} }\widetilde{Z} e^{t \widetilde{A}^T} + (X_0 - \widetilde{X})^{-1}-\widetilde{Z} ]^{-1}  e^{t \widetilde{A}^T}
\end{equation}
where $\widetilde{X}$ is the positive definite solution of the ARE,

\begin{equation}
\label{solexacte2}
A^T\widetilde{X}+\widetilde{X}A-\widetilde{X}BB^T\widetilde{X}+C^TC=0,
\end{equation}
$$\widetilde{A}=A-BB^T\widetilde{X}$$
and $\widetilde{Z}$ is the positive definite solution of the Lyapunov equation
$$\widetilde{A}Z+Z\widetilde{A}^T-BB^T=0$$
\end{theorem}
Unfortunately, the formula  \eqref{solexacte1} is not suitable for large scale problems as it requires the computation of a matrix exponential, of an inverse matrix and various products of matrices.  \\

The paper is organized as follows: In Section 2, we give some basic facts about  the differential Riccati equation and the underlying  finite-horizon LQR problem associated to a dynamical system and its  cost function . In Section 3, we recall  the extended block Arnoldi algorithm with some usefull  classical algebraic properties. Section 4 is devoted to the BDF integration method that allows one to solve numerically differential Riccati equations. In Section 5, we introduce a new approach for the numerical resolution of a differential Riccati equation, based on a projection  onto a sequence of block extended Krylov subspaces. The initial differential Riccati equation is  projected  onto such a subspace to get a low dimensional differential Riccati equation that is solved by the BDF method. We give some theoretical results on the norm of the residual and on  the error. The projection finite horizon LQR problem is studied in Section 6.  In the last section, we give some numerical experiments and also comparisons with different approaches.\\

\noindent Throughout this paper, we use the following notations: The 2-norm
of matrices will be denoted by $ \|\,. \, \| $. .  Finally, $
I_r $ and $ O_{r \times l} $
 will denote the identity of size $ r \times r $ and the zero matrix of size $ r \times l $, respectively.

\section{The finite-horizon LQR problem}

The Linear Quadratic Regulator (LQR) problem is a well known design technique in the  theory of optimal control. The  system dynamics are described by a set of linear differential equations and the cost function  is  a quadratic function. \\
\noindent A linear quadratic regulator (LQR) problem can be described   as follows.  Let $ x(t) $ be the state vector of dimension $ n $, $ u(t) \in \mathbb{R}^p $ the control vector
 and $ y(t)$ the output vector of length $ s $. We consider the following LQR problem  with finite time-horizon (the continuous case ) \cite{abou03,corless,mehrmann91,stillfjord15}: \\

\noindent  For each initial state $x_0$, find the optimal cost $J(x_0,u)$ such that:
\begin{equation} \label{min}
J(x_0,u) =\inf_u \left\{\int_0^{T_f} \left ( y(t)^T\,y(t)+u(t)^T\,u(t) \right )\,dt\right\},
\end{equation}
under the dynamic constrains
\begin{equation}
\label{lti1}
\left\{ \begin{array}{lcl}
 \dot x(t) &=& A\,x(t)+B\,u(t), \;\;\; x(0)=x_0. \\
 y(t) &=& C\,x(t)
 \end{array} 
 \right .
\end{equation}
In addition, when a minimum exists, find an optimal input control  $ \hat u(t)$ which achieves this minimum, that is
\begin{equation}\label{Jmin}
J(x_0,\hat u) = \int_0^{T_f} \left ( \hat y(t)^T\,\hat y(t)+\hat u(t)^T\,\hat u(t) \right )\,dt,
\end{equation}
where the optimal state $\hat x$ and the corresponding output $\hat y$ satisfy \eqref{lti1}.\\
Assuming that the pair $ (A,B) $ is stabilizable (\textit{i.e.} there exists a matrix $ S $
such that $ A-B\,S $ is stable) and the pair $ (C,A) $ is detectable (i.e., $ (A^T,C^T) $
stabilizable),  the optimal  input $ \hat u $ minimizing the functional $ J(x_0,u)$  can be determined through a feedback operator $ K $ such
that the feedback law is given by $ \hat u(t) = K\, \hat x(t) $, where $ K = -B^T\,P(t) $ and $ P(t) \in \mathbb{R}^{n \times n} $
is the unique solution to  the  following differential Riccati equation 
\begin{equation}
\label{ric2}
{\dot P} + A^TP+PA-PBB^TP+C^TC=0;\;\; P(T_f)=0.
\end{equation}
 In addition, the optimal state trajectory satisfies $ \dot {\hat x}(t) = (A-BB^T P(t) )\hat x(t)$ and can also be expressed as
 $$  {\hat x}(t) = e^{tA}  {\hat x}_0 + \displaystyle \int_0^t e^{(t-\tau)A}Bu(\tau)d \tau.$$
 The  optimal cost is given by the following quadratic function of the initial state $x_0$; (see \cite{corless})
\begin{equation}\label{optim}
J(x_0,\hat u)  = x_0^T P(0) x_0.
\end{equation}
We notice that if we set $X(t)=P(T_f-t)$ and $X_0=0$, then $$P(t)=X(T_f-t).$$ Then we recover the solution $X$ to  the  differential Riccati equation \eqref{ric1} and  
$$J(x_0,\hat u)  = x_0^T X(T_f) x_0.$$
For a thorough study on  the existence and uniqueness of the solution of the DRE \eqref{ric1}, see \cite{abou03,astrom,corless}.\\
These results are summarized in the following theorem; (see \cite{corless})
\begin{theorem}
\label{optim}
Assume that the pair $(A,B)$ is stabilizable and the pair $(A,C)$ is detectable. Then, the differential matrix Riccati equation \eqref{ric1} has a unique positive  solution $X$ on $[0,\, T_f]$ and for any initial state $x_0$, the optimal cost of $J(x_0,u)$ is given by
$$J(x_0,\hat u)  = x_0^T X(T_f) x_0,$$
where the optimal control is given by
$$\hat u(t)= -B^TX(T_f-t)\hat x(t),$$
and the optimal trajectory is determined by
$${\dot {\hat x}}(t)=(A-BB^TX(T_f-t))\hat x(t),\;{\rm with}\; \hat x(0)=x_0.$$
\end{theorem}

\begin{remark}
\noindent For the infinite horizon case, the optimal cost is given by 
$$ J(x_0,u_{\infty})=\left \{ \int_0^{\infty} \left ( y(t)^T\,y(t)+u(t)^T\,u(t) \right )\,dt \right\}= x_0X_{\infty} x_0,$$
where $X_{\infty}$ is the unique postitive and stabilizing solution of the algebraic Riccati equation
$$A^T X_{\infty} +  X_{\infty}A -X_{\infty}BB^TX_{\infty} +C^TC=0,$$
and the optimal feedback is given by $u_{\infty}=-B^T X_{\infty} \hat x(t)$.
\end{remark}
\begin{remark}
\noindent For the discrete case, LQR finite-horizon problem is described as follows
\begin{equation} \label{min}
J(x_0,u) =\inf_u \left \{  \displaystyle \sum_0^{N} \left  ( y_k^T\,y_k+u_k^TR\,u_k \right) \right \},
\end{equation}
under the discrete-dynamic constrains
\begin{equation}
\label{lti}
\left\{ \begin{array}{lcl}
 x_{k+1} &=& A\,x_k+B\,u_k. \\
 y_k &=& C\,x_k
 \end{array} 
 \right .
\end{equation}
The optimal control is given by
$$u_k=-F_ku_k,\;{\rm where}\; F_k= (R+B^TZ_{k+1}B)^{-1}B^TZ_{k+1}A,$$
and $P_{k+1}$ is computed by solving the following discrete-time algebraic Riccati equation
$$Z_{k}=A^TZ_{k+1}A-A^TZ_{k+1}B(R+B^TZ_{k+1}B)^{-1}B^TZ_{k+1}A+ C^TC.$$
When $N$ tends to infinity, we obtain the infinite-horizon discrete-time LQR and under some assumptions, $Z_{\infty}=\displaystyle \lim_{k \rightarrow \infty} Z_k$ is the unique positive definite solution to the discrete time algebraic Riccati equation (DARE)
$$Z_ {\infty}=A^{T}Z_{\infty}A-A^{T}Z_{\infty}B\left(R+B^{T}Z_{\infty}B\right)^{-1}B^{T}Z_{\infty}A+C^TC.$$
\end{remark}
In this paper, we consider only the continuous case which needs the development of efficient numerical methods that allow approximate solutions to the related large-scale differential Riccati matrix equation \eqref{ric1} .

\section{The extended block Arnoldi algorithm}
We first recall the extended block Arnoldi process applied to the
pair $ (A,C) $ where $ A \in \mathbb{R}^{n \times n} $ is nonsingular  and $ C \in
\mathbb{R}^{n \times s} $. The projection subspace ${\mathcal
K}^e_k(A,C)$ of $ \mathbb{R}^n $ which is considered in this paper was
introduced in \cite{druskin98,simoncini07}.
$$ {\mathcal
K}^{e}_k(A,C) = Range([C,A^{-1}\,C,A\,C,A^{-2}\,C,A^2\,C,\ldots,A^{-(k-1)}\,C,A^{k-1}\,C]). $$
Note that the subspace ${\mathcal
K}^{e}_k(A,C)$ is a sum of two block Krylov subspaces
$$ \mathcal{K}^e_k(A,C)=\mathcal{K}_{k}(A,C)\, + \, \mathcal{K}_k(A^{-1},A^{-1}C) $$
where $ \mathcal{K}_k(A,C)= Range([A, A\,C,\ldots,A^{k-1}\,C]) $.
The following algorithm allows us to compute an orthonormal basis of the extended Krylov subspace $ {\mathcal
K}^{e}_k(A,C) $. This basis contains information on both $ A $ and $ A^{-1} $. Let $ m $ be some fixed integer which limits the dimension of the constructed basis.  Therefore The extended block Arnoldi process is described as follows:\\

\begin{algorithm}[h!]
\caption{The extended block Arnoldi algorithm (EBA)}\label{eba}
\begin{itemize}
\item $ A $ an $ n \times n $ matrix, $ C $ an $ n \times s $ matrix and $ m $ an integer.
\item Compute the QR decomposition of $ [C,A^{-1}C] $,\textit{ i.e}., $ [C,A^{-1}C] = V_1\Lambda $; \\
\hspace*{1.2cm} Set $ {\mathcal V}_0 = \left [ ~\right] $;
\item For $ j = 1,\ldots,m $
\item \hspace*{0.4cm} Set $ V_j^{(1)} $: first $ s $ columns of $ V_j $ and  $ V_j^{(2)} $: second $ s $ columns of $ V_j $
\item  \hspace*{0.4cm} $ {\mathcal V}_j = \left [ {\mathcal V}_{j-1}, V_j \right ] $; $ \hat V_{j+1} = \left [ A\,V_j^{(1)},A^{-1}\,V_j^{(2)} \right ] $.
\item  Orthogonalize $ \hat V_{j+1} $ w.r.t $ {\mathcal V}_j $ to get $ V_{j+1} $, \textit{i.e.},\\
\hspace*{1.8cm} For $ i=1,2,\ldots,j $ \\
\hspace*{2.4cm} $ H_{i,j} = V_i^T\,\hat V_{j+1} $; \\
\hspace*{2.4cm} $ \hat V_{j+1} = \hat V_{j+1} - V_i\,H_{i,j} $; \\
\hspace*{1.8cm} Endfor $i$
\item Compute the QR decomposition of $ \hat V_{j+1} $, \textit{i.e.}, $ \hat V_{j+1} = V_{j+1}\,H_{j+1,j} $.\\
\item Endfor $j$.
\end{itemize} ${}$ 
\end{algorithm}
Since the above algorithm implicitly involves  a Gram-Schmidt process, the computed block vectors $ {\mathcal V}_m = \left [
V_1,V_2,\ldots,V_m \right ] $, $ V_ i \in \mathbb{R}^{n \times 2s} $ have their columns mutually orthogonal provided none of the upper triangular  matrices $ H_{j+1,j} $ are rank deficient.\\
Hence, after $ m $ steps, Algorithm 1 builds an orthonormal basis $ {\mathcal V}_m $ of the Krylov subspace 
$${\mathcal
K}^{e}_k(A,C) = {\rm Range}(C,A\,C,\ldots,A^{m-1}\,C,A^{-1}\,C,\ldots,(A^{-1})^m\,C) $$
and a block upper Hessenberg matrix $ H_m $ whose non zeros blocks are the $ H_{i,j} $. Note that each submatrix $ H_{i,j} $ ($1 \le i \le j \le m $) is of order $ 2s $. \\
Let $ {\mathcal T}_m \in \mathbb{R}^{2ms \times 2ms} $ be the
restriction of the matrix $ A $ to the extended Krylov subspace $
{\mathcal
K}^{e}_m(A,C) $, \textit{i.e.}, $ {\mathcal T}_m = {\mathcal V}_m^T\,A\,{\mathcal
V}_m $. It is shown in \cite{simoncini07} that $ {\mathcal T}_m $ is
also block upper Hessenberg with $ 2s \times 2s $ blocks.
Moreover, a recursion is derived to compute $ {\mathcal T}_m $ from $
H_m $ without requiring matrix-vector products with $ A $. For
more details about the computation of  $ {\mathcal T}_m $ from $ H_m $, we
refer to \cite{simoncini07}. We note that for large problems, the
inverse of the matrix $ A $ is not computed explicitly. Indeed, in many applications, the nonsingular matrix
$A$ is sparse and structured, allowing an effortless $LU$ decomposition in order to compute the block $\displaystyle{A^{-1}\,V_j^{(2)}}$.  It is also possible to use iterative solvers with preconditioners to solve
linear systems with $ A $. However, when these linear systems are
not solved accurately, the theoretical properties of the extended
block Arnoldi process are no longer valid. The next identities will be of use in the sequel\\
\vspace{0.1cm}
Let $ {\bar {\mathcal T}_m} = {\mathcal V}_{m+1}^T\,A\,{\mathcal V}_m $, and suppose that $ m $ steps of Algorithm 1 have been  run, then we have
\begin{eqnarray}
A\,{\mathcal V}_m & = & {\mathcal V}_{m+1}\,{\bar {\mathcal T}}_m, \\
              & = & {\mathcal V}_m\,{\mathcal T}_m + V_{m+1}\,T_{m+1,m}\,E_m^T.\label{ATVM}
\end{eqnarray}
where $ T_{i,j} $ is the $ 2s \times 2s $ $ (i,j) $ block of $ {\mathcal T}_m $ and $ E_m = [ O_{2s \times 2(m-1)s}, I_{2s} ]^T $ is the matrix of the last $ 2s $ columns of the $ 2ms \times 2ms $ identity matrix $ I_{2ms} $.

\section{The BDF  method for solving  DREs}

In this section, we recall some general facts about the well known BDF method which is a common choice for solving DREs. In the literature, to our knowledge, the integration methods (Rosenbrock, BDF) are directly applied to equation (\ref{ric1}), \cite{arias07,benner04,dieci92}. At each timestep $t_k$, the approximate $X_k$ of the $X(t_k)$, where $X$ is the solution to (\ref{ric1})  is then computed solving an algebraic Riccati equation (ARE) \cite{benner04}.  We consider the general DRE \eqref{ric1} and apply the $p$-step BDF method. At each iteration of the BDF method, the approximation $X_{k+1}$ of  $X(t_{k+1})$ is given by the implicit relation 
\begin{equation}
\label{bdf}
X_{k+1} = \displaystyle \sum_{i=0}^{p-1} \alpha_i X_{k-i} +h \beta {\mathcal F}(X_{k+1}),
\end{equation} 
where $h=t_{k+1}-t_k$ is the step size, $\alpha_i$ and $\beta_i$ are the coefficients of the BDF method as listed  in Table \ref{tab1} and ${\mathcal F}(X)$ is  given by 
$${\mathcal F}(X)= A^T\,X+X\,A-X\,B\,B^T\,X+C^T\,C.$$

\begin{table}[h!!]
\begin{center}
\begin{tabular}{c|cccc} 
\hline
$p$ & $\beta$ &$\alpha_0$ & $\alpha_1$ & $\alpha_2$ \\
\hline
1 & 1 & 1 & &\\
2 & 2/3 & 4/3& -1/3 &\\
3 & 6/11 & 18/11 & -9/11 & 2/11\\
\hline
\end{tabular}
\caption{Coefficients of the $p$-step BDF method with $p \le 3$.}\label{tab1}
\end{center}
\end{table}
\noindent The approximate $X_{k+1}$ solves the following matrix equation
\begin{equation*}
-X_{k+1} +h\beta (C^TC +A^T X_{k+1} + X_{k+1} A- X_{k+1} B B^T X_{k+1}) + \displaystyle \sum_{i=0}^{p-1} \alpha_i X_{k-i} = 0,
\end{equation*}
which can be written as the following  continuous-time algebraic Riccati equation

\begin{equation}
\label{ricbdf}
\mathcal{A}^T\, X_{k+1}  +\,X_{k+1}\, \mathcal{A}  -X_{k+1}\, \mathcal{B}\, \mathcal{B}^T \,X_{k+1} + \mathcal{C}_{k+1}^T \mathcal{C}_{k+1} =0,
\end{equation}
Where, assuming that at each timestep, $X_k$ can be approximated as a product of  low rank factors  $X_{k}\approx Z_{k} Z_k^T$, $Z_k \in \mathbb{R}^{n \times m_k}$, with $m_k \ll n$ (in practice, the $n \times n$ matrices $X_k$ are never computed, as we will explain in a remark after Algorithm 2), the coefficients matrices are given by
$$\mathcal{A}= h\beta A -\displaystyle \frac{1}{2}I,~~ \mathcal{B}= \sqrt{h \beta} B~\mbox{and }\mathcal{C}_{k+1}=[\sqrt{h\beta} C, \sqrt{\alpha_0}Z_k^T,\ldots,\sqrt{\alpha_{p-1}} Z_{k+1-p}^T]^T.$$
These Riccati equations can be solved applying direct methods based on Schur decomposition, or  based on generalized eigenvalues of the Hamiltonian in the small dimensional cases (\cite{arnold84,dooren81,mehrmann91}) or matrix sign function methods (\cite{byers87,kenny92,roberts98}). When the dimension of the problem is large, this approach would be too demanding in terms of computation time and memory. In this case, iterative methods, as Krylov subspaces, Newton-type (\cite{arnold84,benner98,chehab15,kleinman68,lancaster95,guo98}) or ADI-type appear to be a standard choice, see (\cite{bouhamidi16,heyouni09, jbilou03, jbilou06}) for more details. 

\subsection{The BDF+Newton+EBA  method}
As explained in the previous subsection, at each time step $t_k$, the approximation $X_{k+1}$ of $X(t_{k+1})$ is computed solving the large-scale ARE \eqref{ricbdf} which can be expressed as the nonlinear equation 

\begin{equation}
\label{newt1}
{\mathcal F}_k(X_{k+1})=0,
\end{equation}
where $\mathcal{F}_k$ is the matrix-valued function defined by
\begin{equation}
\label{newt2}
{\mathcal F}_k(X_{k+1})={\mathcal A}^TX_{k+1}+X_{k+1} {\mathcal A}-X_{k+1}{\mathcal{B}}{\mathcal{B}}^T X_{k+1}+{\mathcal{C}}_{k+1}^T{\mathcal{C}}_{k+1},
\end{equation}
where  

$X_{k}= Z_{k} Z_k^T$, ${\mathcal{A}}= h\beta A -\displaystyle \frac{1}{2}I$, $ {\mathcal{B}}= \sqrt{h \beta} B$, and $${\mathcal{C}}_{k+1}=[\sqrt{h\beta} C, \sqrt{\alpha_0}Z_k^T,\ldots,\sqrt{\alpha_{p-1}} Z_{k+1-p}^T]^T.$$
In the small dimensional case, direct methods, as Bartel Stewart algorithm is a usual choice for solving the symmetric Riccati equation \eqref{newt1}. For large-scale problems, a common strategy consists in applying the inexact Newton-Kleinman's method combined with an iterative method for the numerical resolution of the large-scale Lyapunov equations arising at each internal iteration of the Newton's algorithm, or a block Krylov projection method directly applied to \eqref{newt1}.\\
Omitting to mention the $k$ index in ${\mathcal F}_k$ in our notations, we define a sequence of approximates to $X_{k+1}$ as follows:

\begin{itemize}
\item Set $X_{k+1}^0=X_k$

\item Build the sequence $\left(X_{k+1}^{p}\right)_{p \in \mathbb{N}}$ defined by 

\begin{equation}
\label{newt3}
X_{k+1}^{p+1}=X_{k+1}^{p}-D{\mathcal F}_{X_{k+1}^{p}}({\mathcal F}(X_{k+1}^{p})
\end{equation}
\end{itemize}
where the Fr\'echet derivative $D{\mathcal F}$ of ${\mathcal F}$ at $X_{k+1}^p$ is given by

\begin{equation}
D{\mathcal F}_{X_{k+1}^{p}}(H)=(\mathcal{A}-\mathcal{B} \, {\mathcal{B}}^T \, X_{k+1}^{p})^T \, H\,+\,H \,(\mathcal{A}-\mathcal{B} \, {\mathcal{B}}^T \,X_{k+1}^{p})
\end{equation} 
A straightforward calculation proves that $X_{k+1}^{p+1}$ is the solution to the Lyapunov equation

\begin{equation}
\label{newt4}
(\mathcal{A}-\mathcal{B} \, {\mathcal{B}}^T \, X_{k+1}^{p})^T \, X\,+\,X \,(\mathcal{A}-\mathcal{B} \, {\mathcal{B}}^T \,X_{k+1}^{p})+X_{k+1}^{p}\, \mathcal{B}\mathcal{B}^T\,X_{k+1}^{p}+\mathcal{C}_{k+1}^T\mathcal{C}_{k+1}\,=0;
\end{equation}
Assuming that all the $X_{k-i}$'s involved in $\mathcal{C}_{k+1}$ can be approximated as  products of low-rank factors, \textit{ie} $X_{k-i}\approx Z_{k-i}Z_{k-i}^T$, we numerically solve \eqref{newt4} applying the Extended-Block-Arnoldi (EBA) method introduced in \cite{heyouni09,simoncini07}.  As mentioned in \cite{benner04}, in order to avoid using complex arithmetics, the constant term of  member of \eqref{newt4} can be splitted into two terms, separating the positive $\alpha_i$'s from the negative ones. This yields a pair of Lyapunov equations $[L_1]$ and $[L_2]$  that have to be numerically solved at each Newton's iteration. Then, the  solution $X_{k+1}^{p+1}$ to \eqref{newt4} is obtained by substracting the solutions of $[L_1]$ and $[L_2]$.
\\
In the numerical examples, we will refer to this method as BDF-Newton-EBA and we will compare its performances against the "reverse" method introduced in the next section.

\section{Projecting and solving the low dimensional DRE } 

In this section, we propose a new approach to obtain low rank approximate solution to the differential Riccati equation \eqref{ric1}. Instead of applying the integration scheme to the original problem (\ref{ric1}), we first reduce the dimension of the problem by projection.\\
Let us apply, under the assumption that the matrix $A$ is nonsingular, the Extended Block Arnoldi (EBA) algorithm to the pair $(A^T,C^T)$, generating the matrices ${\mathcal V}_m$ and ${\bar {\mathcal T}_m} $ as described in section 2. For the sake of simplicity, we omit to mention the time variable $t$ in the formulae. Let $X_m$ be the desired approximate solution to (\ref{ric1}) given as 
\begin{equation}\label{approx1}
X_m = {\mathcal V}_m Y_m {\mathcal V}_m^T,
\end{equation}
satisfying the Galerkin orthogonality condition
\begin{equation}
\label{galerkin}
{\mathcal V}_m^T R_m {\mathcal V}_m =0,
\end{equation}
where $R_m$ is the residual $ R_m = \displaystyle {\dot X_m}-A^T\,X_m-X_m\,A+X_m\,B\,B^T\,X_m- C^T\,C $ associated to the approximation $X_m$.  Then, from \eqref{approx1} and \eqref{galerkin}, we obtain the low dimensional differential Riccati equation
\begin{equation}\label{lowric}
\displaystyle {\dot Y}_m- {\mathcal T}_m\,Y_m-Y_m\,{\mathcal T}_m^T + Y_m\,{B}_m\,{ B}_m^T\,Y_m\, - {C}_m^T  {C}_m=0
\end{equation}
with $ { B}_m = {\mathcal V}_m^T\,B $, $ { C}_m^T = {\mathcal V}_m^T\,C^T = {\mathcal E}_1\,\Lambda_{1,1} $,  where $ {\mathcal E}_1 = [ I_{s} , O_{s \times (2m-1) s}]^T $ is the matrix of the first $ s $ columns of the $ 2ms \times 2ms $ identity matrix $ I_{2ms} $ and $\Lambda_{1,1}$ is the $ s \times s$ matrix obtained from the QR decomposition
\begin{equation}
\label{qrdecomp}
[C^T, A^{-T}C^T]= V_1\, \Lambda \;\; {\rm with}\;\; \Lambda =  \left(
\begin{array}{cc}
\Lambda_{1,1}& \Lambda_{1,2}\\
0&\Lambda_{2,2}
\end{array}
\right).
\end{equation}
At each timestep and for a given Extended Block Krylov projection subspace, applying a BDF($p$) integration scheme, we have to solve  a small dimensional continuous algebraic Riccati equation derived from (\ref{lowric}) as explained in Section 3. This can be done by performing a direct method and we assume that it has a unique symmetric positive semidefinite and stabilizing solution $ Y_m $. Nevertheless, the computations of $ X_m $ and $ R_m $ become increasingly expensive as $ m $ gets larger. In order to stop the EBA iterations, it is desirable to be able to test if $ \parallel R_m
\parallel < \epsilon $, where $\epsilon$ is some chosen tolerance, without having to compute extra matrix  products
involving the matrix $ A $. The next result gives an expression of the residual norm of $ R_m $ which does not require the explicit calculation of the approximate $
X_m $. A factored form will be computed only when the desired accuracy is achieved. This approach will be denoted as Extended Block Arnoldi-BDF($p$) method in the sequel.

\begin{theorem} \label{t2}
Let $ X_m = {\mathcal V}_mY_m{\mathcal V}_m^T $ be the approximation obtained at step $ m $ by the Extended Block Arnoldi-BDF($p$) method and $ Y_m $ solves  the low-dimensional differential Riccati equation (\ref{lowric}),  Then the residual $ R_m $ satisfies
%
%
\begin{equation}
\label{result2}
\parallel R_m \parallel = \parallel T_{m+1,m} \hat Y_m \parallel,
\end{equation}
where $ \hat Y_m $ is the $ 2s \times 2ms $  matrix corresponding to the last $ 2s $ rows of $ Y_m $.
\end{theorem} 
\begin{proof}
From the relations (\ref{approx1}) and (\ref{lowric}), we have
$$ R_m =\displaystyle {\dot X}_m(t)- A^T\,{\mathcal V}_m\,Y_m\,{\mathcal V}_m^T-{\mathcal V}_m\,Y_m\,{\mathcal V}_m^T\,A + {\mathcal V}_m\,Y_m\,{\mathcal V}_m^T\,B\, B^T\,{\mathcal V}_m\,Y_m\,{\mathcal V}_m^T - C^T\,C. $$
Using (\ref{approx1}) and the fact that $ C^T = V_1^{(1)}\,\Lambda_{1,1} $ where $V_1^{(1)}$ is the the matrix of the first $s$ columns of $V_1$ and $\Lambda_{1,1}$ is defined in (\ref{qrdecomp}), we get
\begin{eqnarray*}
R_m & = &  {\mathcal V}_m\displaystyle {\dot Y}_m {\mathcal V}_m^T- ({\mathcal V}_m\,{\mathcal T}_m + V_{m+1}\,T_{m+1,m}\,E_m^T)\,Y_m\,{\mathcal V}_m^T -
          {\mathcal V}_m\,Y_m\,({\mathcal T}_m^T\,{\mathcal V}_m^T + E_m\,T_{m+1,m}^T\,V_{m+1}^T) \\
    &   & +{\mathcal V}_m\,Y_m\,{B}_m\,{B}_m^T\,Y_m\,{\mathcal V}_m^T - V_1^{(1)}\,\Lambda_{1,1}\,\Lambda_{1,1}^T\,{V_1^{(1)}}^T \\
       & = & {\mathcal V}_{m+1}\,{\mathcal J}_m\,{\mathcal V}_{m+1}^T.
\end{eqnarray*}
where $${\mathcal J}_m=\left[ \begin{array}{cc} \displaystyle {\dot Y}_m- {\mathcal T}_m\,Y_m-Y_m\,{\mathcal T}_m^T + Y_m\,{ B}_m\,{B}_m^T\,Y_m\, - {\mathcal E}_1\,\Lambda_{1,1}\,\Lambda_{1,1}^T\,{\mathcal E}_1^T &
             Y_m\,E_m\,T_{m+1,m}^T \\ T_{m+1,m}\,E_m^T\,Y_m & 0 \end{array} \right ].$$
Since $ C_m =  {\mathcal E}_1\,\Lambda_{1,1} $ and $ Y_m $ is  solution of the reduced DRE (\ref{lowric}), then
\begin{eqnarray*}
R_m & = & {\mathcal V}_{m+1} \, \left[ \begin{array}{cc}
          0 & Y_m\,E_m\,T_{m+1,m}^T \\ T_{m+1,m}\,E_m^T\,Y_m & 0 \end{array}\right] \,{\mathcal V}_{m+1}^T 
\end{eqnarray*}
%
%
and
\begin{eqnarray*}
\| R_m\| = \| T_{m+1,m}\,E_m^T\,Y_m\| = \| T_{m+1,m}\,{\hat Y}_m\|,
\end{eqnarray*}
where $ {\hat Y}_m = E_m^T\,Y_m $ represents the $ 2s $ last rows of $ Y_m $.
\end{proof} 

\noindent The result of Theorem \ref{t2} is important in practice, it allows us to stop the iteration when convergence is achieved without computing
the approximate solution $X_m$ at each iteration. In our experiments, this test on the residual is performed at the final timestep. 
We summarize the steps of our method in the following algorithm

\begin{algorithm}[h!]
\caption{The EBA-BDF(p) method for DRE's}\label{algo_EBA_BDF}
\begin{itemize}
\item Input $X_0=X(0)$, a tolerance $tol>0$, an integer $m_{max}$ and timestep  $h$. Set $nbstep=T_f/h$.
\item Compute  $X_1,\, ...,\,X_{p-1}$ as low-rank products $X_j\approx Z_jZ_j^T$ (*)\\
\item  For $ m = 1,\ldots,m_{max} $
\item  Compute an orthonormal basis ${\mathcal V}_m=[V_1,...,V_m]$ of ${\mathcal K}_m^e(A,C^T)=Range[C^T,A^{-1}C,...,A^{-(m-1)}C^T,A^{m-1}C^T]$ 
\item Set ${B}_m={\mathcal V}_m^TB$, ${C}_m^T={\mathcal V}_m^TC^T$ and ${\mathcal T}_m={\mathcal V}_m^TA{\mathcal V}_m$,
\item Apply BDF($p$) to the projected DRE $\dot{Y_m}(t)={\mathcal T}_mY_m(t)+Y_m(t){\mathcal  T}_m^T+Y_m(t) {B}_m\,{B}_m^TY_m(t)+{C}_m^T{C}_m$\\
\item $\rightarrow $ once $Y_m(T_f)$ is computed, if $\|R_m(T_f)\|=\|T_{m+1,m}E_m^T \hat{Y}_m(T_f)\|<tol$, then $m_{max}=m$ 
\item Endfor $i$
\item $X_m(T_f)\approx Z_m(T_f)Z_m^T(T_f)$ (whithout computing the product $X_m(T_f)=V_mY_m(T_f)V_m^T)$ (**)
\end{itemize}  
\end{algorithm}
\noindent Let us give some important remarks on the preceding algorithm:
\begin{itemize}
	\item  In order to initialize the BDF($p$) integration scheme, the $p-1$ approximates $X_1$,...,$X_{p-1}$ are computed by lower-order integration schemes. In our tests, we chose $p=2$ and $X_1$ was computed as a product of low-rank factors ($X_1\approx Z_1Z_1^T$) by the Implicit Euler method BDF($1$).
	\item During the integration process, as explained in Section 3, the constant term ${\mathcal C}_k$ is updated at each timestep $t_k$, taking into account the low-rank factor $\sqrt{(\alpha_{0})}Z_k$ of the approximate factorization $X_k\approx Z_kZ_k^T$. This factorization does not require the computation of $X_k=V_kY_kV_k^T$ as it is obtained by performing a truncated single value decomposition of the small dimensional matrix.  Consider the singular value decomposition of the
 matrix $Y_k=U\, \Sigma\, U^T$ where $\Sigma$ is the
	diagonal matrix of the singular values of $Y_k$ sorted in
	decreasing order. Let $U_l$ be the $2k \times l$ matrix  of   the first $l$ columns of $U$
	corresponding to the $l$ singular values of magnitude greater than
	 some tolerance $dtol$. We obtain the
	truncated singular value decomposition  $Y_k \approx U_l\, \Sigma_l\,
	U_l^T$ where $\Sigma_l = {\rm diag}[\sigma_1, \ldots, \sigma_l]$.
	 Setting $Z_k={\cal V}_k \, U_l\, \Sigma_l^{1/2}$, it
	follows that
	\begin{equation*}
	\label{appXm}
	X_k(t)\approx Z_k(t)\, Z_k(t)^T.
	\end{equation*} 
\end{itemize}

\noindent The following result shows  that the approximation $X_m $ is an exact solution of a perturbed  differential  Riccati equation. \\

\begin{theorem}
Let $X_m$ be the approximate solution given by \eqref{approx1}. Then we have 
\begin{equation}
\label{pertu}
\displaystyle {\dot X}_m=(A-F_m)^T\,X_m+X_m\,(A-F_m)-X_m\,B\,B^T\,X_m+C^T\,C.
\end{equation}
where $ F_m = V_m\,T_{m+1,m}^T\,V_{m+1}^T $.
\end{theorem}
%
\begin{proof}
Let  $ X_m = {\mathcal V}_m\,Y_m\,{\mathcal V}_m^T $ and by  multiplying the reduced-order DRE (\ref{lowric})
on the left by $ {\mathcal V}_m $ and on the right by $ {\mathcal V}_m^T $ and using (\ref{ATVM}), we get 
\begin{eqnarray*} \label{f1}
\displaystyle {\dot X}_m=\left [ A^T{\mathcal V}_m - V_{m+1}T_{m+1,m}E_m^T \right ]Y_m{\mathcal V}_m^T & + & {\mathcal V}_mY_m\left [ A^T{\mathcal V}_m - V_{m+1}T_{m+1,m}E_m^T \right ]^T  \\ \nonumber
- {\mathcal V}_mY_m{\mathcal V}_m^TBB^T{\mathcal V}_mY_m{\mathcal V}_m^T & + & C^TC.
\end{eqnarray*}
The relation (\ref{pertu}) is obtained by letting
 $ F_m = V_m\,T_{m+1,m}^T\,V_{m+1}^T $ and by noticing that $ {\mathcal V}_m\,E_m = V_m $.
\end{proof}

\medskip
\begin{theorem}
Let $X$ be the exact solution of \eqref{ric1} and let $X_m$ be the approximate solution obtained at step $m$. The error $E_m=X-X_m$  satisfies the following equation
\begin{equation}\label{pertu2}
\displaystyle {\dot E}_m=(A^T-XBB^T)E_m+E_m(A-BB^TX)+E_mBB^TE_m +R_m,
\end{equation}
where $R_m$ is the residual given by $R_m=  \displaystyle \dot{X}_m-A^T\,X_m-X_m\,A+X_m\,B\,B^T\,X_m- C^T\,C .$
\end{theorem}
\begin{proof}
The result is easily obtained by subtracting the residual equation from the initial DRE \eqref{ric1}.
\end{proof}

\section{The projected LQR problem}
In this section, we consider the finite horizon LQR problem \eqref{min} and show how the extended block Arnoldi method can be used to give approximate costs to \eqref{min}. 
At step $m$ of the extended block Arnoldi algorithm,  let us consider  the projected low order dynamical system obtained by projecting the original dynamical system \eqref{min} onto the extended block Krylov subspace ${\mathcal K}_m(A^T,C^T) $:
\begin{equation}
\label{ltip}
\left\{ \begin{array}{lcl}
 \dot {\tilde x}_m(t) &=& {\cal T}_m\,{\tilde x}_m(t)+B_m\,{\tilde u}_m(t), \;\;\; {\tilde x}_m(0)=x_{m,0}. \\
{\tilde  y}_m(t) &=& C_m\, {\tilde x}_m(t)
 \end{array} 
 \right .
\end{equation}
where $B_m={\mathcal V}_m^TB$, $C_m^T= {\mathcal V}_m^TC^T={\mathcal E}_1\,\Lambda_{1,1} $ and $x_{m,0}= {\mathcal V}_m^T x_0$. Notice that for small values of the iteration number $m$, $x_m(t)= {\mathcal V}_m {\tilde x}_m(t)$ is an approximation of the original large state $x(t)$ . 
\begin{proposition}
\label{theoJm}
Assume at step $m$ that $({\cal T}_m,B_m)$ is stabilizable and that $({C}_m,{\cal T}_m)$ is detectable and consider  the low dimension  LQR problem  with finite time-horizon:  \\
\noindent Minimize
\begin{equation} \label{minm}
J_m(x_{m,0},{\tilde u}_m) =  \int_0^{T_f} \left ( {\tilde y}_m(t)^T\,{\tilde y}_m(t)+{\tilde u}_m(t)^T\,{\tilde u}_m(t) \right )\,dt,
\end{equation}
under the dynamic constrains \eqref{ltip}.
Then, the unique optimal feedback ${\tilde u}_{*,m}$ minimizing the cost function \eqref{minm} is given by
$${\tilde u}_{*,m}(t)= -B_m {\widetilde Y}_m(t) {\hat x}_m(t),$$
where  ${\widetilde Y}_m(t)=Y_m(T_f-t)$ and $Y_m$  is the unique stabilizing solution of \eqref{lowric}.
\end{proposition}
The proof can be easily derived from the fact that $Y_m$ is obtained from the projected Riccati equation \eqref{lowric} and also from the fact that the dynamical system \eqref{ltip} is obtained from a Galerkin projection of \eqref{lti} onto the same extended block Krylov subspace.\\

\noindent The optimal projected state  satisfies $ \dot {\tilde  x}_m(t) = (A-BB^T {\widetilde Y}_m(t) ){\tilde x}_m(t)$ which can also be expressed as
$$  {\tilde  x}_m(t)= e^{tA}  {\tilde x}_{m,}0 + \displaystyle \int_0^t e^{(t-\tau) {\cal T}_m}B_m {\tilde u}_{*,m}(\tau)d \tau,$$
 and the  optimal cost is given the following quadratic function of the initial state $x_0$
\begin{eqnarray}\label{optimm}
J_m(x_{m,0},{\hat u}_m)  & =  & x_{m,0}^T {\widetilde Y}_m(0) x_{m,0}= x_0{\cal V}_m {\widetilde Y}_m(0){\cal V}_m^Tx_0\\
& = & x_0 P_m(0) x_0=x_0^T X_m(T_f)x_0, 
\end{eqnarray}
since  $X_m(t)=P_m(T_f-t)$ where $X_m={\cal V}_m Y_m {\cal V}_m^T$ is the obtained approximate solution to \eqref{ric1}. This shows clearly that the reduced optimal cost is an approximation of the initial minimal  cost.\\
As already stated, the vector $x_m(t)= {\cal V}_m {\tilde x}_m(t)$ is an approximation of the original state vector $x(t)$. The corresponding feedback law is determined by $u_m(t)=-B^T P_m(t) x_m(t)$ where $P_m(t)$ is the  approximate solution to the differential matrix Riccati equation \eqref{ric2}. Then
\begin{eqnarray*}
J(x_0,u_m) & =  & \inf \left \{  \int_0^{T_f} \left ( x_m(t)^TC^TC\,x_m(t)+u_m(t)^T\,u_m(t) \right )\,dt \right \}\\
& = & \inf \left \{ \int_0^{T_f} \left ( {\tilde x}_m(t)^T {\cal V}_m^T C^TC{\cal V}_m\,x_m(t)+{\tilde x}_m(t)^T {\cal V}_m^T P_m(t)^TBB^TP_m(t)  {\cal V}_m {\tilde x}_m(t)\right )\,dt \right \}.\\
\end{eqnarray*}
Now, using the fact that $P_m(t)={\cal V}_m {\tilde Y}_m(t)
 {\cal V}_m^T$, it follows that
\begin{eqnarray*}
J(x_0,u_m) & = &  \inf \left \{  \int_0^{T_f} \left ( {\tilde y}_m(t)^T\,{\tilde y}_m(t)+{\tilde u}_m(t)^T\,{\tilde u}_m(t) \right )\,dt \right \}\\
& = & J_m(x_{m,0},{\tilde u}_{*,m}),
\end{eqnarray*}
where ${\tilde u}_{*,m}(t)= -B_m {\widetilde Y}_m(t) {\hat x}_m(t)$.
\begin{remark}
It was shown in \cite{corless} that, assuming the pair $(A,B)$ to be stabilizable, ${X(t)}$ is an increasing and symmetric positive set of matrices satisfying the differential Riccati equation \eqref{ric1} and  that there exists a finite positive scalar $M$ such that for any initial state $x_0$, we have 
$$x_0^T X(t)x_0 \le M \parallel x_0 \parallel^2, \; \; \forall t \ge 0,$$ for every initial vector $x_0$ 
which shows that $\parallel X(t) \parallel \le M$  for all $t$. Therefore, ${X(t)}$ converges to a symmetric and  positive matrix   $X_{\infty}$ as $t \rightarrow \infty$: $X_{\infty} = \displaystyle \lim_{t \rightarrow \infty} X(t),$ satisfying the following algebraic Riccati equation (see \cite{corless})
$$  A^TX_{\infty}+X_{\infty}A-X_{\infty}BB^TX_{\infty}+C^TC=0.$$
and $X_{\infty}$ is the only positive and stabilizing solution to this algebraic Riccati equation. 
\end{remark}

\section{Numerical examples}
In this section, in order to assess the interest of projecting the original equation before proceeding to the time integration, we compared the results given by the following approaches :

\begin{itemize}

	\item EBA-BDF($p$) method, as described in Section 5 (Algorithm 2) of the present paper.
	\item BDF($p$)-Newton-EBA:  as described in Section 4.1.
	
\end{itemize}
We also gave an example in which we compare our strategy to the Low-Rank Second-Order Splitting method (LRSOS for short) introduced by Stillfjord \cite{stillfjord15} and to the BDF-LR-ADI method introduced by Benner and Mena \cite{benner04}.
In the first three examples, the number of steps for the BDF($p$) integration scheme was set to $p=2$. 
All the experiments were performed on a laptop with an  Intel Core i7 processor and 8GB of RAM. The algorithms were coded in Matlab R2014b. \\

 \noindent {\bf Example 1}. The matrix $A$  was   obtained from the 5-point discretization of the operators 
 \begin{equation*}
     L_A=\Delta u-f_1(x,y)\frac{\partial u}{\partial x}+ f_2(x,y)\frac{\partial u}{\partial y}+g_1(x,y),
 \end{equation*}
 on the unit square $[0,1]\times [0,1]$ with homogeneous Dirichlet boundary conditions.  The number of inner grid points in each direction is  $n_0$ and the dimension of the matrix $A$ was $n = n_0^2$. Here we set $f_1(x,y) = 10xy$, $f_2(x,y)= e^{x^2y}$, $f_3(x,y) = 100y$, $f_4(x,y)= {x^2y}$ ,  $g_1(x,y) = 20y$ and $g_2(x,y)=x\,y$. The entries of the matrices $B$ and  $C$ were random values uniformly distributed on $[0 \,; 1]$ and the number of columns in $B$ and $C$ was set to $r=s=2$. The initial condition $X_0=X(0)$ was choosen as a low rank product $X_0=Z_0Z_0^T$, where $Z_0\in \mathbb{R}^{n \times 2}$ was randomly generated.
We applied our approach combining a projection onto the Extended Block Krylov subspaces followed by a BDF integration scheme to small to medium size problems, with a tolerance of $10^{-10}$ for the stop test on the residual. In our tests, we used a 2-step BDF scheme with a constant timestep $h$. The first approximate $X_1\approx X(t_0+h)$ was computed by the Implicit Euler scheme and the Extended Block Krylov method for Riccati equation, see \cite{heyouni09} for more details. 
To our knowledge, there is no available exact solution of large scale matrix Riccati differential equations in the literature. In order to check if our approaches produce reliable results, we began comparing our results to the one given by Matlab's ode23s solver which is designed for stiff differential equations. This was done by vectorizing our DRE, stacking the columns of $X$ one on top of each other. This method, based  on Rosenbrock integration scheme, is not suited to large-scale problems. The memory limitation of our computer allowed us a maximum size of $100\times 100$ for the matrix $A$ for the ode23s method.\\

\noindent In Figure \ref{Figure1}, we compared the component $X_{11}$ of the solution obtained by the methods tested in this section, to the solution provided by the ode23s method from matlab,  on the time interval $[0,1]$, for $size(A)=49\times 49$ and a constant timestep $h=10^{-3}$.
 \begin{figure}[H]
	\begin{center}
		\includegraphics[width=15cm,height=7cm]{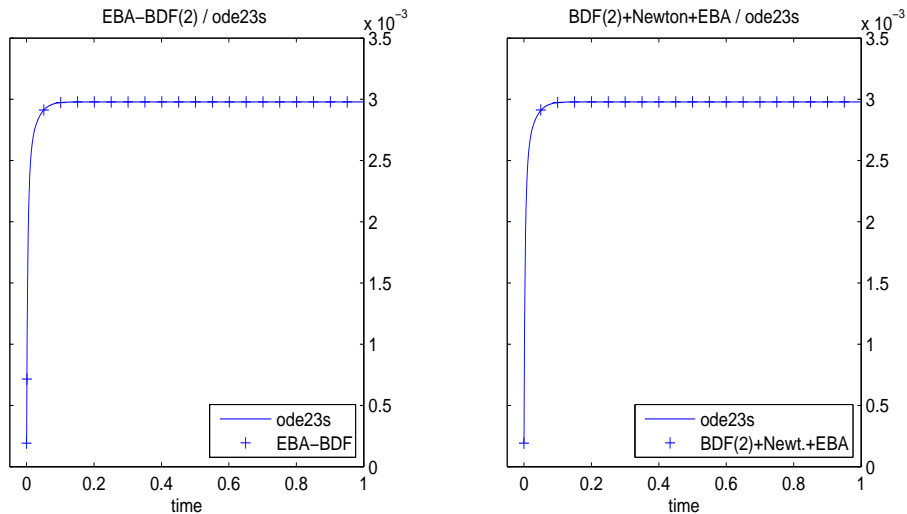}
		\caption{ Values of $X_{11}(t)$ for  $t \in [0,\, 1]$}\label{Figure1}
	\end{center}
\end{figure}
\noindent We observe that all the methods give similar results in terms of accuracy. Figure \ref{Figure2}  features the norm of the difference $X_{EBA}-X_{ode23s}$, where  $X_{EBA}$ is the solution obtained by the EBA-BDF method, at time $t=1$ versus the number of iterations of the Extended Arnoldi process. It shows that the two approximate solutions are very close for a moderate size of the projection space ($m=7$). The ode23s solver took 510 seconds whereas our method gave an approximate solution in $7$ seconds. 
\begin{figure}[h!]
	\begin{center}
\includegraphics[width=10cm,height=7cm]{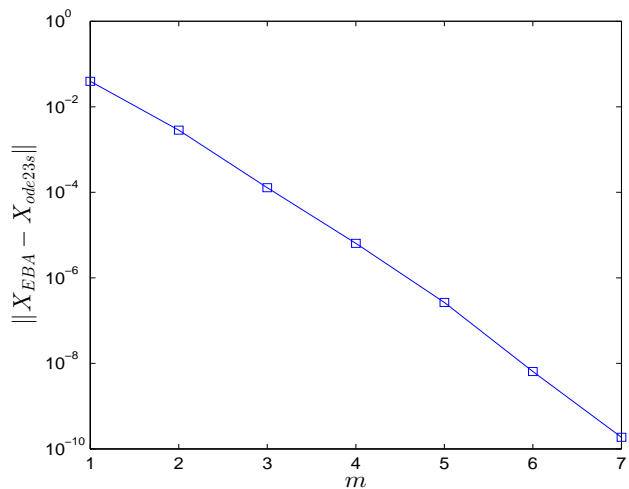}
\caption{Norms of the  errors $\|X_{EBA}(T_f)-X_{ode23s}(T_f)\|_F$ versus the  number of Extended-Arnoldi iterations $m$}\label{Figure2}
\end{center}
\end{figure}
In Table \ref{tab1}, we give  the obtained runtimes in seconds, for the resolution of Equation \eqref{ric1}  for $t \in [0; 1]$, with a timestep $h=0.001$. The figures illustrate that as the dimension of the problem gets larger, it is preferable to reduce the dimension of the problem prior to integration.

\begin{table}[h!]
\begin{center}
\begin{tabular}{c | c c c }
size($A$)&ode23s&EBA+BDF(2)&BDF(2)+Newton-EBA\\
\hline
$49 \times 49$&$30.2$&$6.5$&$37.1$\\
$100 \times 100$&$929.3$&$12.3$&$41.6$\\
$900 \times 900$&$--$&$45.2$&$926.1$\\
\hline
\end{tabular}
\caption{runtimes for ode23s, EBA+BDF(2) and BDF(2)+Newton}\label{tab1}
\end{center}
\end{table}
\noindent {\bf Example 2}.  In this example, we considered the same problem as in the previous example for medium to large dimensions. We kept the same settings as in Example 1, on the time interval  $[0,1]$. For all the tests, we applied the EBA+BDF(2) method and the timestep was set to $h=0.001$. In Table \ref{tab2}, we reported the runtimes, residual norms and the number of the Extended Arnoldi iterations for various sizes of $A$.

\begin{table}[H]
\begin{center}
\begin{tabular}{c | c c c}
size($A$)& Runtime (s)& Residual norms& Number of iterations ($m$)\\
\hline
$100\times100$&$12.3$&$3.1\times 10^{-9}$&$9$\\
$900\times900$&$45.2$&$3.2\times 10^{-8}$&$15$\\
$2500\times2500$&$85.$&$4.8\times 10^{-8}$&$19$\\
$6400\times6400$&$159.7$&$1.8\times 10^{-7}$&$24$\\
$10000\times10000$&$195.3$&$3.7\times 10^{-8}$&$26$\\
\hline
\end{tabular}
\caption{ Execution time (s), residual norms and the  number of Extended-Arnoldi iterations ($m$) for the EBA-BDF(2) method}\label{tab2}
\end{center}
\end{table}
\noindent The next figure shows the norm of the residual $R_m$ versus the number $m$ of Extended Block Arnoldi iterations for the $n=6400$ case.
 \begin{figure}[H]
\begin{center}
\includegraphics[width=10cm,height=7cm]{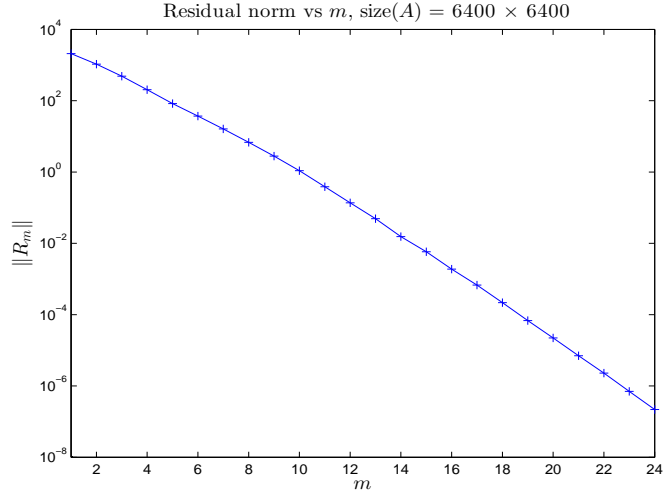}
\caption {Residual norms $\|R_m(Tf)\|$ versus number of Extended-Arnoldi iterations ($m$)}\label{figure3}
\end{center}
\end{figure}
 \noindent As  expected, Figure \ref{figure3} shows that the accuracy improves as $m$ increases. \\

\noindent {\bf Example 3} This example was  taken from  \cite{benner04} and comes from the autonomous linear-quadratic
optimal control problem of one dimensional heat flow
\begin{eqnarray*}
\frac{\partial}{\partial t} x(t,\eta) & = & \frac{\partial^2}{\partial \eta^2} x(t,\eta)+b(\eta) u(t)   \\
x(t,0) & = & x(t,1)=0, t>0\\
x(0,\eta) & = & x_0(\eta), \eta \in [0,1]\\
y(x) & = & \int_0^1 c(\eta) x(t,\eta) d \eta, x>0.
\end{eqnarray*}
Using a standard finite element approach based on the first order B-splines, we obtain the following ordinary differential equation
\begin{eqnarray}\label{ode1}
\left\{ \begin{array}{lcl}
M \dot {\tt x}(t)& = & K {\tt x}(t) + F u(t)\\
y(t) & =& C {\tt x}(t),
\end{array} \right.
\end{eqnarray}
where the matrices $M$ and $K$ are given by:
$$M=\frac{1}{6n}\left(
\begin{array}{ccccc}
4&1&&&\\
1&4&1&&\\
&\ddots&\ddots&\ddots\\
&&1&4&1\\
&&&1&4
\end{array}
\right), \;\; K=-\alpha\,n\,\left(
\begin{array}{ccccc}
2&-1&&&\\
-1&2&-1&&\\
&\ddots&\ddots&\ddots\\
&&-1&2&-1\\
&&&-1&2
\end{array}
\right).$$
Using the semi-implicit Euler method, we get  the following discrete dynamical system
$$(M-\Delta t K)\, \dot{x}(t)= M\, x(t)+ \Delta t\, F u_k.$$
We set $A=-(M-\Delta t K)^{-1}\, M$ and $B=\Delta t\, (M-\Delta t K)^{-1}\, F$. The entries of the   $ n \times s$  matrix $ F $ and the $s \times n$ matrix  $ C $ were random values uniformly distributed on $ [0,1] $. We chose the initial condition as $X_0=0_{n\times n}=Z_0Z_0^T$, where $Z_0=O_{n\times 2}$. In our experiments we used $s=2$,   $\Delta t = 0.01$   and  $\alpha=0.05$. 
 \\
In Figure \ref{figure6}, we plotted the first component $X_{11}$ over the time interval $[0,2]$ of the approximate solutions obtained by the EBA-BDF(2),  the BDF(2)-Newton-EBA methods  and the ode23s solver. It again illustrates the similarity of the results in term of accuracy. 

 \begin{figure}[H]
	\begin{center}
		\includegraphics[width=10cm,height=7cm]{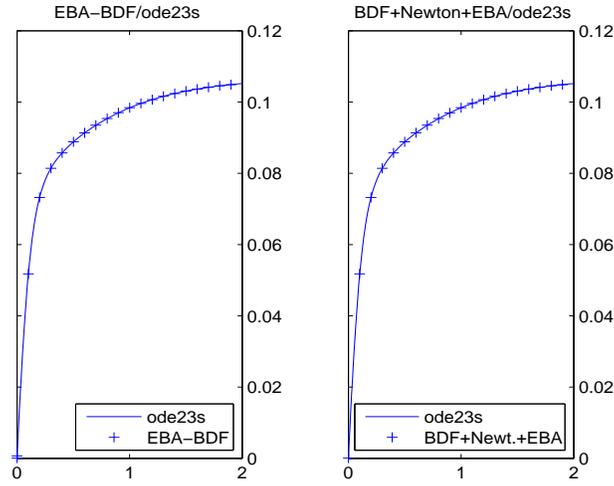}
		\caption{Computed values of $X_{11}(t)$,  $t \in [0,\,2]$ for ode23s and EBA-BDF (left) and for ode23s with BDF(2)+Newton+EBA (right) with  $size(A)=49\times 49$}\label{figure6}
	\end{center}
\end{figure}
Figure \ref{figure7} illustrates the  remark of Proposition \ref{theoJm} in Section 6, claiming that under certain conditions, we have  $\displaystyle{ \lim_{t \rightarrow \infty} X(t)=X_{\infty}}$, where $X_{\infty}$ is the only positive and stabilizing solution to the ARE: $A^TX+XA-XBB^TX+C^TC=0$. In order to do so, we plotted the norm of the difference $X_{EBA}(t)-X_{\infty}$ in function of the time parameter $t$  on a relatively large time scale ($t\in [0,50]$), for an initial value $X_0=0_n$, where $X_{EBA}$ is the matrix computed by the EBA-BDF(2) method. In this example, we have $\|X_{EBA}(50)-X_{\infty}\|\approx 4\times 10^{-5}$. 

 \begin{figure}[H]
	\begin{center}
		\includegraphics[width=10cm,height=7cm]{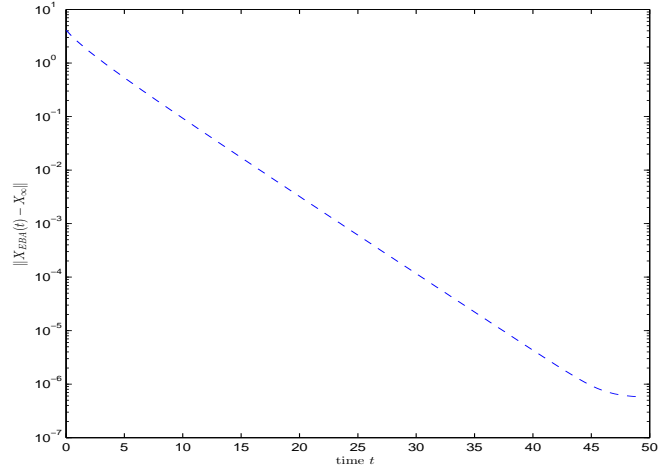}
		\caption{ Error norms $\parallel X_{EBA-BDF}(t)-X_{\infty}\parallel$ for $t \in [0,\, 50]$ with  $size(A)=400\times 400$.}\label{figure7}
	\end{center}
\end{figure}
\noindent Table \ref{tab3} reports on  the computational times, residual norms and number of Krylov iteration,  for various sizes of the coefficient matrices on the time interval $[0,1]$.  We only reported  the results given for the EBA-BDF(2) method which clearly outperfoms the BDF(2)-Newton EBA method in terms of computational time. For instance,  in the case where $size(A)=1600 \times 1600$ case, the BDF-Newton EBA method  required $521$ seconds and more than $5000$ seconds for $size(A)=2500 \times 2500$.

\begin{table}[H]
	\begin{center}
		\begin{tabular}{c | c c c}
			size($A$)&Runtime (s)&Residual norms&number of iterations ($m$)\\
			\hline
			$1600\times1600$&$14.2$&$3.2\times 10^{-12}$&$10$\\
			$2500\times2500$&$17.1$&$7\times 10^{-12}$&$9$\\
			$4900\times4900$&$25.6$&$1.3\times 10^{-11}$&$9$\\
			$6400\times6400$&$42.2$&$8.5\times 10^{-12}$&$10$\\
			$10000\times10000$&$143.9$&$4.5\times 10^{-11}$&$8$\\
			\hline
		\end{tabular}
		\caption{EBA-BDF(2) method: runtimes (s), residual norms, number of Arnoldi iterations ($m$).}\label{tab3}
	\end{center}
\end{table}

\noindent {\bf Example 4} In this last example, we applied the EBA-BDF(1) method to the well-known problem Optimal Cooling of Steel Profiles. The matrices were extracted from the IMTEK collection \footnote{https://portal.uni-freiburg.de/imteksimulation/downloads/benchmark}. We compared the EBA-BDF(1) method to the splitting method LRSOP (in its Strang splitting variant) \cite{stillfjord15} and the BDF(1)-LR-ADI method \cite{benner04}, for sizes $n=1357$ and $n=5177$, on the time interval $[0\,;5]$. The initial value $X_0$ was choosen as $X_0=Z_0Z_0^T$, with $Z_0=0_{n\times 1}$. The code for the LRSOP method contains paralell loops which used 4 threads on our machine for this test.  The BDF(1)-LR-ADI consists in applying the BDF(1) integration scheme to the original DRE \eqref{ric1}. As for the BDF(p) Newton-EBA method presented in Section 4.1, it implies the numerical resolution of a potentially large-scale algebraic Riccati equation for each timestep. This resolution was performed via a the LR-ADI method, available in the M.E.S.S Package \footnote{http://www.mpi-magdeburg.mpg.de/projects/mess/}. For the LRSOS and BDF(1)-LR-ADI, the tolerance for column compression strategies were set to $10^{-8}$. The timestep was set to $h=0.01$ and the tolerance for the Arnoldi stop test was set to $10^{-7}$ for the EBA-BDF(1) method and the projected equations were numerically solved by a dense solver ({\tt care} from matlab) every 3 extended Arnoldi iterations.
\vspace{0.2cm}

\begin{table}[H]
	\begin{center}
		\begin{tabular}{c | c c c c c}
			size $n$&EBA-BDF(1) & LRSOS&BDF(1)-LR-ADI&  LRSOS vs EBA-BDF(1) &EBA-BDF(1) vs BDF(1)-LR-ADI  \\
			\hline
			 	$1357$& $130.4$ s&$145.5$ s&$259.2$ s&$3.9\times 10^{-4}$&$7.2\times 10^{-10}$\\
			\hline
			 $5177$& $578$ s&$919.4$ s&$6482.3$ s&$5.2\times 10^{-4}$&$5.5\times 10^{-8}$\\
			\hline
			
		\end{tabular}
		\caption{Execution times for the Optimal Cooling of Steel Profiles }\label{tab4}
	\end{center}
\end{table}
\noindent In Table \ref{tab4},  we listed the obtained runtimes and  the relative  errors $\|X_{EBA-BDF{1}}(t_f)-X_{*}(t_f)\|_F\,/\,\|X_{EBA-BDF(1)}(t_f)\|_F$,  where $*$ denotes LRSOP or BDF(1)-LR-ADI, at final time $t_f=5$.
\noindent  For the EBA-BDF(1), the number $m$ of Arnoldi iterations and the residual norms were $m=18$, $\|R_m\|=1.1\times 10^{-7}$ for $n=1357$ and $m=27$, $\|R_m\|=5.4\times 10^{-6}$ for $n=5177$.  This example shows that our method EBA-BDF is interesting in terms of execution time. 

\section{Conclusion}
In this article, we have proposed a strategy consisting in projecting a large-scale differential problem onto a sequence of extended block Krylov subspaces and solving the projected matrix differential equation by a BDF integration scheme. The main purpose of this work was to show that projecting the original problem before integration gives encouraging results. In the case where the matrix $A$ is not invertible (or ill-conditioned), which is out of the scope of this paper,  one could  contemplate the use of the block-Arnoldi algorithm instead of the extended-block-Arnoldi method. We gave some theoretical results  such as a simple expression of the residual which does not require the computation of products of large matrices and  compared our approach experimentally to the common approach for which the integration scheme is directly applied to the large-scale original problem. Our experiments have shown that projecting before performing the time integration  is interesting in terms of computational time.   

\vspace{0.5cm}

\noindent \textbf{Acknowledgements} We would like to thank Dr. Tony Stillfjord for providing us with the codes of the low-rank second-order splitting (LRSOP) method and for his insightfull comments.

\bibliographystyle{plain}
\end{document}